\documentclass[psamsfonts]{amsart}
\usepackage{amssymb}
\usepackage[left=1.15in,right=1.15in,top=1.0in,bottom=0.87in]{geometry}

\numberwithin{equation}{section}

\newtheorem{em-deff}{Definition}[section]
\newtheorem{lemma}[em-deff]{Lemma}
\newtheorem{theorem}[em-deff]{Theorem}
\newtheorem{corollary}[em-deff]{Corollary}

\newtheorem{proposition}[em-deff]{Proposition}
\newtheorem{em-fact}[em-deff]{Fact}
\newtheorem{em-example}[em-deff]{Example}
\newtheorem{claim}[em-deff]{Claim}
\newtheorem{problem}[em-deff]{Problem}

\newtheorem{em-remark}[em-deff]{Remark}

\newenvironment{example}{\begin{em-example} \em }{ \end{em-example}}
\newenvironment{remark}{\begin{em-remark} \em }{ \end{em-remark}}
\newenvironment{deff}{\begin{em-deff} \em }{ \end{em-deff}}

\newcommand{\N}{\mathbb N}
\newcommand{\Z}{\mathbb Z}

\def\ent{\mathrm{ent}}
\def\supp{\mathrm{supp}}
\def\Per{\mathrm{Per}}

\def\Aut{\mathrm{Aut}}

\input xy
\xyoption{all}

\title{Algebraic entropy of shift endomorphisms on abelian groups}

\author[M. Akhavin]{Maryam Akhavin}
\address[Maryam Akhavin]{Faculty of Mathematical Sciences, University for Teacher Education \\
599 Taleghani Ave., Tehran 15614, Iran}
\email{m$\_$akhavin@tmu.ac.ir}

\author[F. Ayatollah Zadeh Shirazi]{Fatemah Ayatollah Zadeh Shirazi}
\address[Fatemah Ayatalloh Zadeh Shirazi]{Faculty of Mathematics, Statistics and Computer Science, College of Science, University of Tehran\\
Enghelab Ave., Tehran, Iran}
\email{fatemah@khayam.ut.ac.ir}

\author[D. Dikranjan]{Dikran Dikranjan}
\address[Dikran Dikranjan]{Universit\`a di Udine, Dipartimento di Matematica e Informatica
\\ via delle Scienze, 206 - 33100 Udine, Italy}
\email{dikran.dikranjan@dimi.uniud.it}

\author[A. Giordano Bruno]{Anna Giordano Bruno}
\address[Anna Giordano Bruno]{Universit\`a di Udine, Dipartimento di Matematica e Informatica
\\ via delle Scienze, 206 - 33100 Udine, Italy}
\email{anna.giordanobruno@dimi.uniud.it}

\author[A. Hosseini]{Arezoo Hosseini}
\address[Arezoo Hosseini]{Department of Mathematics, Faculty of Science, University of Guilan \\
Manzarieh Ave., Rasht, Iran}
\email{a$\_$hosseini@guilan.ac.ir}

\dedicatory{Dedicated to Prof. A. Chademan, with best wishes for him}

\begin{document}

\begin{abstract}
For every finite-to-one map $\lambda:\Gamma\to\Gamma$ and for every abelian group $K$, the {\em generalized shift} $\sigma_\lambda$ of the direct sum $\bigoplus_\Gamma K$ is the endomorphism defined by $(x_i)_{i\in\Gamma}\mapsto(x_{\lambda(i)})_{i\in\Gamma}$ \cite{AKH}. In this paper we analyze and compute the algebraic entropy of a generalized shift, which turns out to depend on the cardinality of $K$, but mainly on the function $\lambda$. We give many examples showing that the generalized shifts provide a very useful universal tool for producing counter-examples.
\end{abstract}

\maketitle

We denote by $\mathbb Z$, $\mathbb P$, and $\mathbb N$ respectively the set of integers, the set of primes, and the set of natural numbers; moreover $\mathbb N_0=\N\cup\{0\}$.
For a set $\Gamma$, $\mathcal P_{\rm fin}(\Gamma)$ denotes the family of all finite subsets of $\Gamma$.
For a set $\Lambda$ and an abelian group $G$ we denote by $G^\Lambda$ the direct product $\prod_{i\in\Lambda}G_i$, and by $G^{(\Lambda)}$ the direct sum $\bigoplus_{i\in\Lambda} G_i$, where all $G_i=G$. 
For a set $X$, $n\in\N$, and a function $f:X\to X$ let $\Per(f)$ be the set of all periodic points and $\Per_n(f)$ the set of all periodic points of order at most $n$ of $f$ in $X$.

\section{Introduction}
The measure entropy was introduced by Kolmogorov and Sinai in ergodic theory in the mid fifties of the last century. 
Some ten years later  Adler, Konheim, and McAndrew \cite{AKM} introduced the notion of topological entropy $h_{top}(T)$ of a continuous self-map $T: X \to X$ of a compact topological space $X$. A prominent example in both cases is provided by the \emph{Bernoulli shifts}. Since these shifts are the core of this paper, we introduce them here in full detail.

\begin{example}
Let $K$ be a non-trivial finite group with neutral element $e_K$. 
\begin{itemize}
\item[(a)]  The \emph{two-sided Bernoulli shift} $\overline{\beta}_K$ of the group $K^{\Z}$ is  defined by 
$$\overline\beta_K(\ldots, x_0,x_1,x_2, \ldots)=(\ldots, x_{-1}, x_0, x_1, \ldots),\ (\mbox{i.e.},\  \overline\beta_K((x_n)_{n\in\Z})=(x_{n-1})_{n\in\Z}, \mbox{ for } (x_n)_{n\in\Z}\in K^{\Z}).$$
\item[(b)] The \emph{right  Bernoulli shift} $\beta_K$ \ (respectively, \emph{left Bernoulli shift} $_K\beta$) of the group $K^{\N_0}$ is  defined by 
$$
\beta_K(x_1,x_2,x_3,\ldots)=(e_K,x_1,x_2,\ldots), \; (\mbox{respectively}, \;\; _K\beta(x_0,x_1,x_2,\ldots)=(x_1,x_2,x_3,\ldots).
$$ 
\end{itemize}

The standard product measure of the compact group $K^{\Z}$  (respectively, $K^{\N_0}$) coincides with its Haar measure and $\overline{\beta}_K$ (respectively, $_K\beta$)  is a measure-preserving continuous automorphism (respectively, endomorphism) with topological entropy $\log |K|$ coinciding with the measure entropy. This explains their relevance to both ergodic theory and topological dynamics. 

The right Bernoulli shift $\beta_K$ of $K^{\N_0}$  is less relevant in this respect for two reasons:  it is {\em not} measure-preserving (so not relevant for ergodic theory) and its topological entropy is 0. 
\end{example} 

A possible definition of algebraic entropy  for endomorphisms of abelian groups was briefly mentioned in \cite{AKM}. 
Later on, in 1975 in \cite{W} Weiss defined the algebraic entropy as follows: let $G$ be an abelian group and $F$ be a finite subgroup of $G$; for an endomorphism $\phi:G\rightarrow G$ and $n\in\N$, let $T_n(\phi,F):=F+\phi(F)+\ldots+\phi^{n-1}(F)$ be the \emph{$n$-trajectory} of $F$ with respect to $\phi$. The \emph{algebraic entropy of $\phi$ with respect to $F$} is $$H(\phi,F):={\lim_{n\to +\infty}\frac{\log|T_n(\phi,F)|}{n}},$$ and the \emph{algebraic entropy} of $\phi:G\to G$ is $$\ent(\phi)=\sup\{H(\phi,F): F\ \text{is a finite subgroup of } G\}.$$

Since the definition is based on finite subgroups $F$, and in particular $F$ is contained in the torsion part $t(G)$ of $G$, the algebraic entropy depends only on the restriction of $\phi$ on $t(G)$, that is $\ent(\phi)=\ent(\phi\restriction_{t(G)})$.
The basic properties of the algebraic entropy can be found in \cite{DGSZ,W}. The most relevant of them, known also as 
Addition Theorem, can be found in \S \ref{background} (Theorem \ref{AT}), which collects all relevant properties of the algebraic entropy used in this paper. 

As far as the algebraic entropy is concerned,  the right Bernoulli shift restricted to the direct sum $K^{(\N_0)}$ turned out to be more relevant (while the restriction of the left Bernoulli shift $_K\beta\restriction_{K^{(\N_0)}}$ has algebraic entropy $0$, see Example \ref{NewExample}). More precisely, for a non-trivial finite abelian group $K$ the restriction $\beta_K\restriction_{K^{(\N_0)}}$ has entropy  $\log|K|$ \cite[Example 1.9]{DGSZ} and 
one can show that every function $f$ defined on all endomorphisms of torsion abelian groups with values in the extended non-negative reals and satisfying $f(\beta_{\Z(p)}\restriction_{\Z(p)^{(\N_0)}})=\log |p|$, the Addition Theorem and 
a few other natural properties (namely, Lemmas \ref{conjugation_by_iso}, \ref{log_law} and Remark \ref{lemma1} (b)) must necessarily coincide with the algebraic entropy $\ent(-)$ \cite[Theorem 6.1]{DGSZ}.

With the aim of computing the entropy of other endomorphisms of abelian groups, in this paper we consider a modification of the generalized shifts, introduced in \cite{AKH}. 

\begin{deff}\cite{AKH}
For a non-empty set $\Gamma$, an arbitrary map $\lambda:\Gamma\rightarrow\Gamma$  and an abelian group  $K$ the \emph{generalized shift} $\sigma_{\lambda,K}: K^\Gamma\to K^\Gamma$ is defined by $(x_i)_{i\in\Gamma}\mapsto(x_{\lambda(i)})_{i\in\Gamma}$ for every $(x_i)_{i\in\Gamma}\in K^\Gamma$.
\end{deff}

When there is no possibility of confusion we write $\sigma_\lambda$ instead of $\sigma_{\lambda,K}$.
In case $|K|>1$, the subgroup $K^{(\Gamma)}$ of $K^\Gamma$ is $\sigma_{\lambda,K}$-invariant if and only if $\lambda$ has finite fibers (see Lemma \ref{lemma2}), and it is possible to consider the restriction  $\sigma_{\lambda,K}^\oplus=\sigma_{\lambda,K}\restriction_{K^{(\Gamma)}}$ of  $\sigma_{\lambda,K}$ to $K^{(\Gamma)}$. 
Again, when there is no possibility of confusion we write $\sigma_{\lambda,K}^\oplus$ simply as $\sigma_{\lambda}^\oplus$, $\sigma_{\lambda,K}$ or just $\sigma_\lambda$.

There is a close relation between the Bernoulli shifts and the generalized shifts. For example, the left Bernoulli shift and the two-sided Bernoulli shift are generalized shifts (see Examples \ref{NewExample} and \ref{theorem4} (d) respectively), while the right Bernoulli shift $\beta_{K}\restriction_{K^{(\N_0)}}$ cannot be obtained as a  generalized shift $\sigma_{\lambda}^\oplus$ from any function $\lambda:\N_0\to\N_0$. 
Nevertheless, it can be ``approximated" quite well by the generalized shift $\sigma_{\psi}^\oplus$ of $K^{(\N_0)}$ induced by the map $\psi: \N _0 \to \N_0$ defined by $\psi(i)={i-1}$ for every $i>0$ and $\psi({0})={0}$. Indeed, both $\sigma_{\psi}^\oplus$ and   $\beta_{K}\restriction_{K^{(\N_0)}}$ leave invariant the finite-index subgroup $H=K^{(\N)}$  and $\sigma_{\psi}^\oplus\restriction_H =\beta_{K}\restriction_H$ (in particular, they have the same entropy $\log|K|$). 
 
In this paper we compute the entropy of an arbitrary generalized shift $\sigma_{\lambda}^\oplus:K^{(\Gamma)}\to K^{(\Gamma)}$. 
More precisely, we show that  $\ent(\sigma_{\lambda}^\oplus)$, depends only on combinatorial properties of the map $\lambda$ and the cardinality of $K$.
To prove this we analyze the structure of the map $\lambda$ and, more specifically, its (iterated) counter-images
(since, in some sense, the iterations of $\lambda$ and the iterations of $\sigma_\lambda$ ``go in opposite directions'').
Roughly speaking we decompose the generalized shift $\sigma_{\lambda}^\oplus$ in ``independent elementary shifts'' (as the generalized shift $\sigma_{\psi}^\oplus$ considered above), which have the same algebraic entropy as the right Bernoulli shift $\beta_K\restriction_{K^{(\N_0)}}$, and the number $s_\lambda$ of these independent elementary shifts, multiplied by $\log|K|$, gives precisely $\ent(\sigma_{\lambda}^\oplus)$ (see Theorem \ref{theorem6}).

\bigskip
\noindent\textbf{Convention.}
From now on we assume that $\Gamma$ is a non-empty set, $\lambda:\Gamma\rightarrow\Gamma$ is an arbitrary map, and $K$ is a non-trivial finite abelian group.
We denote by $G_\Gamma$ the group $K^{(\Gamma)}$ and for a subset $A$ of $\Gamma$, we identify $K^{(A)}$ with the subgroup $\{x\in G_\Gamma:\supp(x)\subseteq A\}$ of $G_\Gamma$, and we denote $K^{(A)}$ by $G_A$. In case $A=\emptyset$ we assume that $G_\emptyset=\{0\}$.
We denote the generalized shift based on $G_\Gamma$ and $\lambda:\Gamma\to\Gamma$ simply by $\sigma_{\lambda}$, writing it in some cases $\sigma_{\lambda,K}$, when we need to specify the group.

\section{Background on algebraic entropy}\label{background}

We start collecting basic results on algebraic entropy, mainly from \cite{DGSZ,W}.

\begin{lemma}\label{conjugation_by_iso}\emph{\cite[Proposition 1.2]{W}}
Let $G$, $H$ be abelian groups and $\phi:G\to G$, $\eta:H\to H$ endomorphisms. If there exists an isomorphism $\xi:G\to H$ such that $\phi=\xi^{-1}\eta\xi$, then $\ent(\phi)=\ent(\eta)$.
\end{lemma}

\begin{lemma}\label{log_law}\emph{\cite[Proposition 1.3]{W}}
Let $G$ be an abelian group and $\phi:G\to G$ an endomorphism. Then $\ent(\phi^k) = k\,\ent(\phi)$ for every non-negative integer $k$. If $\phi$ is an
automorphism, then $\ent(\phi^k) = |k|\ent(\phi)$ for every $k\in \Z$.
\end{lemma}

The following is one of the main results on algebraic entropy.

\begin{theorem}[Addition Theorem]\label{AT}\emph{\cite[Theorem 3.1]{DGSZ}}
Let $G$ be a torsion abelian group, $\phi:G\to G$ an endomorphism and $H$ a $\phi$-invariant subgroup of $G$. If $\overline{\phi}:G/H\to G/H$ is the endomorphism induced on the quotient by $\phi$, then $\ent(\phi)=\ent(\phi\restriction_H)+\ent(\overline\phi)$.
\end{theorem}

\begin{remark}\label{lemma1} Let $G$ be an abelian group and $\phi:G\to G$ an endomorphism.
\begin{itemize} 
\item[(a)] A particular case of the above theorem was proved in \cite[Proposition 1.4]{W}: if $n\in\N$, $G=\bigoplus_{i=1}^nG_i$  and $G_i$ is a $\phi$-invariant subgroup of $G$ for $i= 1,\ldots, n$, 
then $\ent(\phi)=\sum_{i=1}^n\ent(\phi\restriction_{G_i})$.
\item[(b)] If the group $G$ is a direct limit of $\phi$-invariant subgroups $\{G_i:i\in I\}$, then ${\ent}(\phi)=\sup_{i\in I}\ent(\phi {\restriction_{G_i}})$ \cite[Proposition 1.6]{DGSZ}. 
\item[(c)] Using (b), one can extend (a) to arbitrary direct sums $G=\bigoplus_{i\in I}G_i$.
\end{itemize}
\end{remark}

\begin{lemma}\label{locnilp}\emph{\cite{DGSZ}}
Let $G$ be an abelian group and $\phi:G\to G$ an endomorphism.
\begin{itemize}
	\item[(a)]If $X$ is a set of generators of $G$ and for every $x\in X$ there exists $d_x\in\N$ such that $\phi^{d_x}(x)=0$, then $\ent(\phi)=0$. 
	\item[(b)]If $\phi$ is periodic, then $\ent(\phi)=0$.
\end{itemize}
\end{lemma}

\section{Strings and an effective equivalence relation}

Now we introduce a notion that will play a prominent role in the computation of the algebraic entropy of the generalized shifts. 

\begin{deff}
\begin{itemize}
	\item[(a)] A \emph{string} of $\lambda$ (in $\Gamma$) is an infinite sequence of pairwise distinct elements $S=\{m_t\}_{-t\in\N_0}$ such that $\lambda(m_t)=m_{t+1}$ for every $-t\in\N$. 
	\item[(b)] Let $s_\lambda:=\sup\{|\mathcal F|: \mathcal F\ \text{is a family of pairwise disjoint strings of $\lambda$} \}$, and
	\item[(c)] $\Gamma^+:=\bigcap_{n=1}^\infty\lambda^n(\Gamma)$.
\end{itemize}
\end{deff}

A string $S=\{m_t\}_{-t\in\N_0}$ of $\lambda$ in $\Gamma$ is said to be \emph{acyclic} if $\lambda^n(m_0)\not\in S$ for every $n\in\N$. The next claim is easy to prove.

\begin{claim}\label{acyclic}
Each string $S$ of $\lambda$ in $\Gamma$ contains an acyclic string $S'$ of $\lambda$.
\end{claim}

The importance of $\Gamma^+$ consists in the fact that it contains all strings of $\lambda$ as well as all periodic points of $\lambda$. Obviously, $\Gamma^+ = \Gamma$ if and only if $\lambda$ is surjective.  In general, the restriction $\lambda\restriction_{\Gamma^+}: \Gamma^+\to \Gamma^+$ need not be surjective (but this holds true if $\lambda$ has finite fibers).

\medskip
Consider the following equivalence relation: $i\Re_\lambda j$ in $\Gamma$ if and only if there exist $m,n\in\N_0$ such that $\lambda^n(i)=\lambda^m(j)$.
Let $\alpha_\lambda:=|\{i/\Re_\lambda\in\Gamma/\Re_\lambda:i/\Re_\lambda\ \text{contains at least one string of $\lambda$}\}|$. Obviously, $\alpha_\lambda\leq s_\lambda$.

\begin{example}\label{theorem3-first_part}
Suppose that $\lambda$ is injective.
\begin{itemize}
    \item[(a)] The relation $\Re_\lambda$ in this particular case becomes: $\Re_\lambda=\{(i,j)\in\Gamma\times\Gamma:\exists m\in\Z:
i=\lambda^m(j)\}$.
    \item[(b)]The relation $\Re_\lambda$ has three types of equivalence classes: 
    \begin{itemize}
	\item[(b$_1$)] finite equivalence classes, 
	\item[(b$_2$)] infinite equivalence classes contained in $\Gamma^+$ (i.e., containing a string of $\lambda$), of the form 
	$$\{\ldots,\lambda^{-1}(i),i,\lambda(i),\lambda^2(i),\ldots\}.$$
	\item[(b$_3$)] infinite equivalence classes non-contained in $\Gamma^+$ (i.e., non-containing a string of $\lambda$), of the form $\{i,\lambda(i),\lambda^2(i),\ldots\}$ with $i\in\Gamma\setminus\lambda(\Gamma)$.
	\end{itemize}
    \item[(c)] Then $\alpha_\lambda$ is the number of the infinite equivalence classes in (b$_2$). Consequently $\alpha_\lambda=s_\lambda$.
\end{itemize} 
\end{example}

\begin{example}\label{example5-firstpart}
Let $\Gamma=\N_0$. For every $n\in\N$ let $\varphi_n:\Gamma\to\Gamma$ be defined by
\begin{equation*}
\varphi_n(m)=
\begin{cases}
0 & \text{if}\ m=0,1,\ldots,n, \\
m-1 & \text{otherwise}.
\end{cases}
\end{equation*}
The diagram for $\varphi_n$ with $n>1$ is the following:
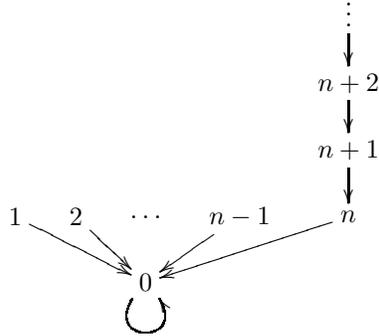
\begin{equation*}
\xymatrix@-1pc{
  &   &        &     & \vdots \ar@{->}[d] \\
  &   &        &     & n+2 \ar@{->}[d] \\
  &   &        &     & n+1 \ar@{->}[d]\\
1 \ar@{->}[drr] & 2 \ar@{->}[dr] & \ldots & n-1 \ar@{->}[dl] & n \ar@{->}[dll]  \\
  &   & 0\ar@(dl,dr)[]    &  & \\
}
\bigskip
\end{equation*}
In this case $s_{\varphi_n}=\alpha_{\varphi_n}=1$ and $\lambda(\Gamma)=\Gamma^+=\Gamma\setminus\{1,\ldots,n-1\}$.

\medskip
For every $n\in\N$, let $\psi_n:\Gamma\to\Gamma$ be defined by
\begin{equation*}
\psi_n(m)=
\begin{cases}
0 & \text{if}\ m=0,1,\ldots,n, \\
(k-1)n+i & \text{if}\ m=kn+i\ \text{with}\ 0\leq i<n\ \text{and}\ k\in\N.
\end{cases}
\end{equation*}
The diagram for $\psi_n$ is the following:
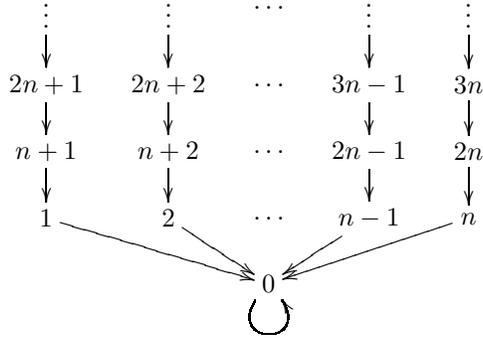
\begin{equation*}
\xymatrix@-1pc{
\vdots \ar@{->}[d]& \vdots \ar@{->}[d] & \ldots & \vdots \ar@{->}[d] & \vdots \ar@{->}[d] \\
2n+1 \ar@{->}[d] & 2n+2 \ar@{->}[d] & \ldots & 3n-1 \ar@{->}[d] & 3n \ar@{->}[d] \\
n+1 \ar@{->}[d]  & n+2 \ar@{->}[d]  & \ldots & 2n-1 \ar@{->}[d] & 2n \ar@{->}[d] \\
1 \ar@{->}[drr]  & 2 \ar@{->}[dr]   & \ldots & n-1 \ar@{->}[dl] & n \ar@{->}[dll] \\
                 &                  & 0\ar@(dl,dr)[]&                  & \\
}
\bigskip
\end{equation*}
For this function $s_{\psi_n}=n$, $\alpha_{\psi_n}=1$ and $\Gamma^+=\Gamma$. Note that $\psi_1=\varphi_1$.

\medskip
Let $\Gamma=\N_0\times\N_0$ and $\lambda_0:\Gamma\to\Gamma$ be
defined by
\begin{equation*}
\lambda_0(m,k)=
\begin{cases}
(0,0) & \text{if}\ m=k=0, \\
(m-1,0) & \text{if}\ k=0\ \text{and}\ m\in\N, \\
(m,k-1) & \text{if}\ m\in\N_0\ \text{and}\ k\in\N.
\end{cases}
\end{equation*}
The diagram for $\lambda_0$ is the following:
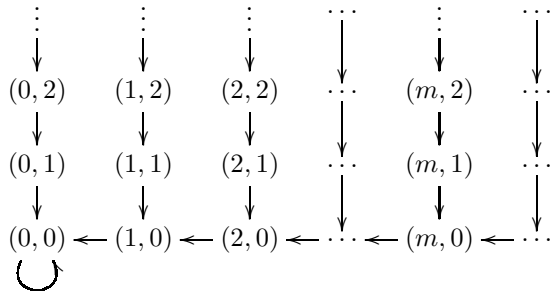
\begin{equation*}
\xymatrix@-1pc{
\vdots\ar[d] & \vdots\ar[d] & \vdots\ar[d] & \ldots\ar[d] & \vdots\ar[d] & \ldots\ar[d] \\
(0,2)\ar[d] & (1,2)\ar[d] & (2,2)\ar[d] & \ldots\ar[d] & (m,2)\ar[d] & \ldots\ar[d] \\
(0,1) \ar[d] & (1,1)\ar[d] & (2,1)\ar[d] & \ldots\ar[d] & (m,1)\ar[d] & \ldots\ar[d] \\
(0,0)\ar@(dl,dr)[] & (1,0)\ar[l] & (2,0)\ar[l] & \ldots\ar[l] & (m,0)\ar[l] & \ldots\ar[l] \\
}
\bigskip
\end{equation*}
In this case $s_{\lambda_0}=\omega$, $\alpha_{\lambda_0}=1$ and $\Gamma^+=\Gamma$.
\end{example}

\section{The entropy of the generalized shift}

\begin{remark}\label{subgroup}
Let $\mu:\Gamma\to\Gamma$ be a function. If $H$ is a subgroup of an abelian group $L$, then $H^{\Gamma}$ is a $\sigma_{\mu,L}$-invariant subgroup of $L^\Gamma$. Moreover, $\sigma_{\mu,L}\restriction_{H^{\Gamma}}=\sigma_{\mu,H}:H^\Gamma\to H^\Gamma$.
Analogously, if $L^{(\Gamma)}$ is a $\sigma_{\mu,L}^\oplus$-invariant subgroup of $L^\Gamma$, then $H^{(\Gamma)}$ is a $\sigma_{\mu,L}^\oplus$-invariant subgroup of $L^{(\Gamma)}$, and $\sigma_{\mu,L}^\oplus\restriction_{H^{(\Gamma)}}=\sigma_{\mu,H}^\oplus:H^{(\Gamma)}\to H^{(\Gamma)}$.
\end{remark}

\begin{claim}\label{G_F}
Let $x\in G_\Gamma$ and $F=\supp(x)$. Then for every $m\in\N$:
\begin{itemize}
	\item[(a)]$\supp(\sigma_\lambda^m(x))=\lambda^{-m}(F)$;
	\item[(b)]$\sigma_\lambda^m(G_F)\leq G_{\lambda^{-m}(F)}$;
	\item[(c)]$T_m(\sigma_\lambda,G_F)\leq G_{F\cup\lambda^{-1}(F)\cup\ldots\cup\lambda^{-m+1}(F)}$.
\end{itemize}
\end{claim}
\begin{proof}
(a) If $y=\sigma_\lambda(x)$, then $y_i=x_{\lambda(i)}\neq0$ if and only if $\lambda(i)\in F$, that is, $i\in\lambda^{-1}(F)$, and so $\supp(y)=\lambda^{-1}(F)$. Proceeding by induction it is possible to prove that $\supp(\sigma_\lambda^m(x))=\lambda^{-m}(F)$ for every $m\in\N$.

(b) Follows from (a).

(c) By (b) $T_m(\sigma_\lambda,G_F)\leq G_F+G_{\lambda^{-1}(F)}+\ldots+ G_{\lambda^{-m+1}(F)}$ and $G_F+G_{\lambda^{-1}(F)}+\ldots+ G_{\lambda^{-m+1}(F)}\subseteq G_{F\cup\lambda^{-1}(F)\cup\ldots\cup\lambda^{-m+1}(F)}$.
\end{proof}

The next lemma shows the relevance of our following assumption on $\lambda$ of having finite fibers.

\begin{lemma}\label{lemma2}
The following conditions are equivalent:
\begin{itemize}
\item[(a)] $\lambda^{-1}(i)$ is finite for each $i\in\Gamma$;
\item[(b)] $\sigma_{\lambda,K}(G_\Gamma)\subseteq G_\Gamma$;
\item[(c)] $\sigma_{\lambda,L}(L^{(\Gamma)})\subseteq L^{(\Gamma)}$ for every non-trivial abelian group $L$.
\end{itemize}
\end{lemma}
\begin{proof}
(b)$\Rightarrow$(a) Let $i\in\Gamma$ and $x\in G_{\{i\}}\setminus\{0\}$. By Claim \ref{G_F}(a) $\supp(\sigma_\lambda(x))=\lambda^{-1}(i)$ and by the assumption $\sigma_\lambda(x)\in G_\Gamma$, hence $\lambda^{-1}(i)$ is finite.

A similar argument shows that (a)$\Rightarrow$(c) and (c)$\Rightarrow$(b) is obvious.
\end{proof}

\bigskip
\noindent\textbf{Convention.}
From now on we suppose that $\lambda$ has finite fibers, that is, $\lambda^{-1}(i)$ is finite for every $i\in\Gamma$. 

\bigskip
\begin{proposition}\label{composition}
Let $L$ be an abelian group with at least two elements and let $\mu,\nu:\Gamma\to\Gamma$ be functions with finite fibers. For $\sigma_\mu,\sigma_\nu:L^\Gamma\to L^\Gamma$ and $\sigma_{\mu}^\oplus,\sigma_{\nu}^\oplus:L^{(\Gamma)}\to L^{(\Gamma)}$:
\begin{itemize}
    \item[(a)] $\sigma_\mu\circ\sigma_\nu=\sigma_{\mu\circ\nu}$ and $\sigma_{\mu}^\oplus\circ\sigma_{\nu}^\oplus=\sigma_{\mu\circ\nu}^\oplus$ (hence $\sigma_\mu^m=\sigma_{\mu^m}$ and $(\sigma_{\mu}^\oplus)^m=\sigma_{\mu^m}^\oplus$ for every $m\in\N$).
    \item[(b)] \emph{\cite{AKH}} The following conditions are equivalent:
\begin{itemize}
	\item[(b$_1$)]$\mu$ is injective (respectively, surjective);
	\item[(b$_2$)]$\sigma_\mu$ is surjective (respectively, injective);
	\item[(b$_3$)]$\sigma_{\mu}^\oplus$ is surjective (respectively, injective).
\end{itemize}   
    \item[(c)] In particular, the following conditions are equivalent:
\begin{itemize}
	\item[(c$_1$)]$\mu$ is bijective;
	\item[(c$_2$)]$\sigma_\mu$ is an automorphism of $L^\Gamma$;
	\item[(c$_3$)]$\sigma_{\mu}^\oplus$ is an automorphism of $L^{(\Gamma)}$.
\end{itemize}    
    In this case, $(\sigma_\mu)^{-1}=\sigma_{\mu^{-1}}$ and $(\sigma_{\mu}^\oplus)^{-1}=\sigma_{\mu^{-1}}^\oplus$.
\end{itemize}
\end{proposition}

Note that the equivalences (b$_1$)$\Leftrightarrow$(b$_2$) and (c$_1$)$\Leftrightarrow$(c$_2$) hold without any assumption on the fibers of $\mu$ and $\nu$.

\begin{corollary}\label{G_Fd}
For every $m\in\N$, $\ker\sigma_\lambda^m=G_{\Gamma\setminus\lambda^m(\Gamma)}$.
\end{corollary}
\begin{proof}
It suffices to prove that $\ker\sigma_\lambda=G_{\Gamma\setminus\lambda(\Gamma)}$ and then apply Proposition \ref{composition}(a). If $x\in\ker\sigma_\lambda$, equivalently $\supp(\sigma_\lambda(x))=\emptyset$. By Claim \ref{G_F}(a) $\supp(\sigma_\lambda(x))=\lambda^{-1}(\supp(x))$. Then $\supp(\sigma_\lambda(x))=\lambda^{-1}(\supp(x))=\emptyset$ if and only if $\supp(x)\cap\lambda(\Gamma)=\emptyset$. This is the same as $\supp(x)\subseteq \Gamma\setminus \lambda(\Gamma)$, that is, $x\in G_{\Gamma\setminus \lambda(\Gamma)}$.
\end{proof}

The next lemma gives a characterization (in terms of $\lambda$) of the $\sigma_\lambda$-invariance of the subgroups $G_A$ of $G_\Gamma$.

\begin{lemma}\label{invariance}
If $A\subseteq \Gamma$, then $G_A$ is $\sigma_\lambda$-invariant if and only if $\lambda^{-1}(A)\subseteq A$. If $A$ is also $\lambda$-invariant, then $\sigma_\lambda\restriction_{G_A}=\sigma_{\lambda\restriction_A}$.
\end{lemma}
\begin{proof}
The condition $\sigma_\lambda(G_A)\subseteq G_A$ is equivalent to $\sigma_\lambda(G_i)\subseteq G_A$ for every $i\in A$, that is, $\lambda^{-1}(i)\subseteq A$ for every $i\in A$, which is equivalent to $\lambda^{-1}(A)\subseteq A$.
Assume now that $\lambda^{-1}(A)\cup\lambda(A)\subseteq A$. Then it is possible to consider both $\sigma_\lambda\restriction_{G_A}$ and $\sigma_{\lambda\restriction_A}$. It is clear that they coincide on $G_A$.
\end{proof}

Lemma \ref{invariance} shows that in case $\lambda^{-1}(A)\subseteq A$ for $A\subseteq \Gamma$, it is possible to consider $\sigma_\lambda\restriction_{G_A}:G_A\to G_A$.

\begin{remark}\label{relambda}
We see here that we can assume that for the relation $\Re_\lambda$ there exists only one equivalence class in $\Gamma$ (so coinciding with the whole $\Gamma$). 

\smallskip
Indeed, if $i/\Re_\lambda$ is a generic equivalence class, then $i/\Re_\lambda\supseteq \lambda(i/\Re_\lambda)\cup\lambda^{-1}(i/\Re_\lambda)$. By Lemma \ref{invariance} 
\begin{equation}\label{Relambda}
\sigma_\lambda\restriction_{i/\Re_\lambda}=\sigma_{\lambda\restriction_{i/\Re_\lambda}}.
\end{equation}
Let now $R$ be a set of representing elements of $\Re_\lambda$ in $\Gamma$. Then $G_\Gamma=\bigoplus_{i\in R}G_{i/\Re_\lambda}$. By Remark \ref{lemma1}(c) and \eqref{Relambda} $$\ent(\sigma_\lambda)=\sum_{i\in R}\ent (\sigma_\lambda\restriction_{G_{i/\Re_\lambda}})=\sum_{i\in R}\ent(\sigma_{\lambda\restriction_{i/\Re_\lambda}}),$$ and so we can assume that $R$ is a singleton.
\end{remark}

The next result gives the very useful formula \eqref{magic}, which is applied in the proof of the main theorem.

\begin{remark}\label{add-rem}
Let $\Gamma=\Gamma'\cup\Gamma''$ be a partition of $\Gamma$. Then $\lambda^{-1}(\Gamma')\subseteq \Gamma'$ if and only if $\lambda(\Gamma'')\subseteq\Gamma''$. Suppose that these equivalent conditions hold. By Lemma \ref{invariance} $G_{\Gamma'}$ is $\sigma_\lambda$-invariant. 
\begin{itemize}
    \item[(a)]  Let $p_2:G_\Gamma=G_{\Gamma'}\oplus G_{\Gamma''}\to G_{\Gamma''}$ and $\pi:G_\Gamma\to G_\Gamma/G_{\Gamma'}$ be the canonical projections.
    Denote by $\xi:G_\Gamma/G_{\Gamma'}\to G_{\Gamma''}$ the (unique) isomorphism  such that $p_2=\xi\circ\pi$.   Finally, let 
  $\overline{\sigma_\lambda}:G_\Gamma/G_{\Gamma'}\to G_\Gamma/G_{\Gamma'}$ be the homomorphism induced by $\sigma_\lambda$. Then $\overline{\sigma_\lambda}=\xi^{-1}\sigma_{\lambda\restriction_{\Gamma''}}\xi$.
To better explain the situation, this means that the following diagram commutes:
\begin{equation*}
\xymatrix@-1pc{
G_\Gamma \ar@{->}[dr]_{\pi} \ar@{->}[rr]^{p_2} & & G_{\Gamma''} \ar[rr]^{\sigma_{\lambda\restriction_{\Gamma''}}}& & G_{\Gamma''} \\
 & G_\Gamma/G_{\Gamma'} \ar@{->}[ur]_{\xi} \ar[rr]^{\overline{\sigma_\lambda}}& & G_\Gamma/G_{\Gamma'} \ar@{->}[ur]_{\xi}\\
}
\end{equation*}
By Lemma \ref{conjugation_by_iso} $\ent(\overline{\sigma_\lambda})=\ent(\sigma_{\lambda\restriction_{\Gamma''}})$.
   \item[(b)] By (a) and Theorem \ref{AT} 
    \begin{equation}\label{magic}
    \ent(\sigma_\lambda)=\ent(\sigma_{\lambda}\restriction_{G_{\Gamma'}})+\ent({\sigma_{\lambda\restriction_{\Gamma''}}}).
    \end{equation}
     Therefore,
	\begin{equation}\label{magic2}
	\ent(\sigma_\lambda)\geq\ent(\sigma_{\lambda}\restriction_{G_{\Gamma'}})\ \text{and}\ \ent(\sigma_\lambda)\geq\ent({\sigma_{\lambda\restriction_{\Gamma''}}}).
	\end{equation}
\end{itemize}
\end{remark}

The next corollary shows that $\ent(\sigma_\lambda)$ can be computed from its restriction to $\Gamma^+\subseteq\Gamma$.

\begin{corollary}\label{corollary}
\begin{itemize}
	\item[(a)]$\ent(\sigma_\lambda)=\ent(\sigma_{\lambda\restriction_{\Gamma^+}})$.
	\item[(b)]In particular, for each $k\in\N$, $\ent(\sigma_\lambda)=\ent(\sigma_{\lambda\restriction_{\lambda^k(\Gamma)}})$.
	\item[(c)]If there exists a non-empty finite subset $F$ of $\Gamma$ such that $\Gamma=\bigcup_{n\in\N_0}\lambda^{n}(F)$, then $\Gamma^+=\Per(\lambda)$ and consequently $\ent(\sigma_\lambda)=0$.
\end{itemize}
\end{corollary}
\begin{proof}
(a) Since $\lambda(\Gamma^+)\subseteq \Gamma^+$, by \eqref{magic} in Remark \ref{add-rem}(b) $\ent(\sigma_\lambda)=\ent(\sigma_\lambda\restriction_{G_{\Gamma\setminus\Gamma^+}})+\ent(\sigma_{\lambda\restriction_{\Gamma^+}})$. We prove that $\ent(\sigma_\lambda\restriction_{G_{\Gamma\setminus\Gamma^+}})=0$. Let $x\in G_{\Gamma\setminus\Gamma^+}$. Then $\supp(x)\subseteq \Gamma\setminus\Gamma^+$. For every $i\in\supp(x)$ there exists $h(i)\in\N$ such that $i\not\in\lambda^{h(i)}(\Gamma)$, and so $\lambda^{-h(i)}(i)=\emptyset$. Let $h(x)=\max\{h(i):i\in\supp(x)\}$. By Claim \ref{G_F}(a) $\supp(\sigma_\lambda^{h(x)}(x))=\lambda^{-h(x)}(\supp(x))$, which is empty, and so $\sigma_\lambda^{h(x)}(x)=0$. By Lemma \ref{locnilp}(a) $\ent(\sigma_\lambda\restriction_{G_{\Gamma\setminus\Gamma^+}})=0$.

(b) Follows from (a) since $\ent(\sigma_\lambda)\geq\ent(\sigma_{\lambda\restriction_{\lambda^k(\Gamma)}})\geq\ent(\sigma_{\lambda\restriction_{\Gamma^+}})$, where \eqref{magic2} in Remark \ref{add-rem}(b) is applied twice.

(c) Clearly $\Gamma^+\supseteq\Per(\lambda)$. Let $i\in\Gamma^+$. For every $n\in\N$ there exists $i_n\in\Gamma$ such that $\lambda^n(i_n)=i$. Moreover, for $n\in\N$ there exist $m_n\in\N$ and $j_{m_n}\in F$ such that $\lambda^{m_n}(j_{m_n})=i_n$ and so $\lambda^{n+m_n}(j_{m_n})=i$. Since $F$ is finite, there exists $j\in F$
such that for a strictly increasing sequence $({n_k})_{k\in\N}$  in $\N$ one has $j_{m_{n_k}}=j$ for all $k\in\N$. 
Then $\lambda^{n_k+m_{n_k}}(j)=i$ for all $k\in\N$. Choose $k\in\N$ such that $n_k > n_1 + m_{n_1}$, then $n_k + m_{n_k} > n_1 + m_{n_1}$ as well. 
Therefore, $i=\lambda^{ n_1+m_{n_1}}(j)=\lambda^{n_k + m_{n_k}}(j)$ yields  $i\in\Per(\lambda)$.
\end{proof}

\begin{example}\label{NewExample} Let $\lambda: \N_0 \to \N_0$ be defined by $\lambda(n) = n+1$. Then $\sigma_\lambda: K^{(\N_0)}\to K^{(\N_0)}$ coincides with the left Bernoulli shift $_K\beta$ of $K^{(\N_0)}$. Since $\N_0^+=\emptyset$ (in the notation of Corollary \ref{corollary}), we deduce  from Corollary \ref{corollary}(c) that $\ent(\sigma_\lambda)=0$. 
\end{example}

Another application of Corollary \ref{corollary} is the following example, in which $\Gamma$ is a compact metric space.

\begin{example}\label{example19}
Let $(\Gamma,d)$ be a compact metric space and $\lambda:\Gamma\rightarrow\Gamma$ be a contraction (i.e., $d(\lambda(x),\lambda(y)) < d(x,y)$  for every pair of distinct points $x,y$ of $\Gamma$) such that for each $i\in\Gamma$, $\lambda^{-1}(i)$ is finite. Then $\ent(\sigma_\lambda)=0$.

Indeed, using a standard compactness argument one proves that $\lambda(\Gamma^+)=\Gamma^+\ne \emptyset$. Moreover, $|\Gamma^+|=1$ since otherwise
there exist $x,y\in \Gamma^+$ such that $d(x,y)=\mathrm{diam}\, \Gamma^+$. By $\lambda(\Gamma^+)=\Gamma^+$ there exist $x',y'\in \Gamma^+$ with $\lambda(x') = x$ and $\lambda(y') =y$. 
So $\mathrm{diam}\, \Gamma^+= d(x,y)= d(\lambda(x'),\lambda(y')) < d(x',y')$, a contradiction. Therefore, $|\Gamma^+|=1$, so Corollary \ref{corollary}(a) applies.
\end{example}

\subsection{The main theorem}

For a set $X$, in what follows we use the following notation: $$|X|^*:=
\begin{cases}
|X| & \text{in case}\ |X|\ \text{is finite},\\
+\infty & \text{in case}\ |X|\ \text{is infinite}.
\end{cases}$$

\begin{deff}
Two strings $S$ and $S'$ of $\lambda$ in the set $\Gamma$ are said to be \emph{strongly disjoint} if $S$ and $\lambda^n(S')$ are disjoint for every $n\in\N_0$ and $S'$ and $\lambda^n(S)$ are disjoint for every $n\in\N_0$. 
\end{deff}

By the definition of string we immediately have the following result.

\begin{claim}\label{strongly-disjoint}
If $s_\lambda = n\in\N_0$ is finite, then $\Gamma$ contains $n$ strings of $\lambda$ that are pairwise strongly disjoint.
\end{claim}

The next is our main theorem, which calculates the entropy of a generalized shift $\sigma_\lambda:G_\Gamma\to G_\Gamma$, proving that it depends only on the function $\lambda:\Gamma\to\Gamma$ and on the cardinality of $K$.

\begin{theorem}\label{theorem6}
Let $\Gamma$ be a set, $\lambda:\Gamma\to\Gamma$ a function such that $\lambda^{-1}(i)$ is finite for every $i\in\Gamma$, and consider $\sigma_\lambda:G_\Gamma\to G_\Gamma$, where $G_\Gamma=K^{(\Gamma)}$ and $K$ is a non-trivial finite abelian group. Then $$\ent(\sigma_\lambda)=|s_\lambda|^*\log|K|.$$
\end{theorem}
\begin{proof}
By Corollary \ref{corollary}(a) and Remark \ref{relambda} we can assume without loss of generality that $\Gamma=\Gamma^+$ and that there exists only one equivalence class for $\Re_\lambda$.

\medskip
Suppose that $s_\lambda=n$ for some $n\in\N_0$.
If $n>0$, there exist $n$ pairwise strongly disjoint acyclic strings of $\lambda$ $$S_1:=\{m^1_t\}_{-t\in\N_0},\ldots,S_n:=\{m^n_t\}_{-t\in\N_0}$$ in $\Gamma$ by Claims \ref{acyclic} and \ref{strongly-disjoint}.

Let $\Psi=\emptyset$ if $n=0$ and $\Psi:=S_1\cup\ldots\cup S_n$ otherwise. Then $\lambda^{-1}(\Psi)\subseteq\Psi$, since $\Gamma=\Gamma^+$, and so $G_\Psi$ is $\sigma_\lambda$-invariant by Lemma \ref{invariance}. Let $\nu:=\lambda\restriction_{\Gamma\setminus\Psi}$. 
By \eqref{magic} in Remark \ref{add-rem}(b) $$\ent(\sigma_\lambda)=\ent(\sigma_\lambda\restriction_{G_\Psi})+\ent(\sigma_{\nu}).$$ 
For $F=\{\lambda(m^1_0),\ldots,\lambda(m^n_0)\}$, $\Gamma\setminus\Psi=\bigcup_{n\in\N_0}\nu^n(F)$, since the strings are acyclic. By Corollary \ref{corollary}(c) $\ent(\sigma_{\nu})=0$ and so
\begin{equation}\label{eq2}
\ent(\sigma_\lambda)=\ent(\sigma_\lambda\restriction_{G_\Psi}).
\end{equation}
Now $\Psi$ is disjoint union of $S_1,\ldots, S_n$ and so $G_\Psi=G_{S_1}\oplus\ldots\oplus G_{S_n}$. Moreover, $\lambda^{-1}(S_j)\subseteq S_j$ since $\Gamma=\Gamma^+$, and so by Lemma \ref{invariance} $G_{S_j}$ is $\sigma_\lambda$-invariant for each $j\in\{1,\ldots,n\}$. By Remark \ref{lemma1}(b)
\begin{equation}\label{eq3}
\ent(\sigma_\lambda\restriction_{G_\Psi})=\ent(\sigma_\lambda\restriction_{G_{S_1}})+\ldots+\ent(\sigma_\lambda\restriction_{G_{S_n}}).
\end{equation}
For every $j\in\{1,\ldots,n\}$, since $\sigma_\lambda\restriction_{G_{S_j}}$ is precisely the right Bernoulli shift $\beta_K\restriction_{G_{S_j}}:G_{S_j}\to G_{S_j}$,
\begin{equation}\label{eq4}
\ent(\sigma_\lambda\restriction_{G_{S_j}})=\log|K|.
\end{equation}

By (\ref{eq2}), (\ref{eq3}) and (\ref{eq4}) $\ent(\sigma_\lambda)=n\log|K|=s_\lambda\log|K|$.

\medskip
Assume now that $|s_\lambda|^*=+\infty$. Then $s_\lambda>n$ for every $n\in\N$. Fix $n\in\N$. There exist $n$ pairwise disjoint strings of $\lambda$ $$S_1:=\{m^1_t\}_{-t\in\N_0},\ldots,S_n:=\{m^n_t\}_{-t\in\N_0}$$ in $\Gamma$. Define
\begin{equation*}
\Lambda_1:=S_1\cup\{\lambda^s(m^1_0):s\in\N\}, \ldots,\Lambda_n:=S_n\cup\{\lambda^s(m^n_0):s\in\N\}\ \text{and}\ \Lambda:=\Lambda_1\cup\ldots\cup\Lambda_n.
\end{equation*}
Note that $\lambda(\Lambda)\subseteq \Lambda$, so that we can consider the map $\lambda\restriction_\Lambda: \Lambda \to \Lambda$. 
By \eqref{magic2} in Remark \ref{add-rem}(b)
\begin{equation}\label{eq5}
\ent(\sigma_\lambda)\geq\ent(\sigma_{\lambda\restriction_\Lambda}).
\end{equation}
Since $s_{\lambda\restriction_\Lambda}=n$, by the finite case of the proof of the theorem applied to the map $\lambda\restriction_\Lambda$, $\ent(\sigma_{\lambda\restriction_\Lambda})=n\log|K|$. By (\ref{eq5}) $\ent(\sigma_\lambda)\geq n\log|K|$, and this holds true for every $n\in\N$, so that $\ent(\sigma_\lambda)=+\infty=|s_\lambda|^*\log|K|$.
\end{proof}

We see now a first application of our main theorem.

\begin{example}\label{theorem4}\label{example5}\label{theorem3}
\begin{itemize}
	\item[(a)]In general $\ent(\sigma_\lambda)\geq|\alpha_\lambda|^*\log|K|$. Indeed, by the definitions $s_\lambda\geq\alpha_\lambda$ and so apply Theorem \ref{theorem6}.
	\item[(b)]The inequality in (a) can be strict: consider the functions of Example \ref{example5-firstpart}. In all three cases there exists just one equivalence class and so one equivalent class containing at least one infinite string of $\lambda$. Then $\alpha_{\varphi_n}=\alpha_{\psi_n}=\alpha_{\lambda_0}=1$, but $\ent(\sigma_{\varphi_n})=\log|K|$, $\ent(\sigma_{\psi_n})=n\log|K|$ and $\ent(\sigma_{\lambda_0})=+\infty$.
	\item[(c)] If $\lambda:\Gamma\to \Gamma$ is injective, the inequality in (a) becomes an equality, since in this case $\alpha_\lambda=s_\lambda$: by Theorem \ref{theorem6} and Example \ref{theorem3-first_part} $\ent(\sigma_\lambda)=|\alpha_\lambda|^*\log|K|=|s_\lambda|^*\log|K|$.
	\item[(d)] For $\lambda: \Z\to \Z$ defined by $\lambda(n)=n-1$ for every $n\in\Z$, the generalized shift $\sigma_\lambda$ coincides with the two-sided shift $\overline{\beta}_K\restriction_{K^{(\Z)}}$
	of $K^{(\Z)}$. Since $\alpha_\lambda= s_\lambda= 1$, one obtains from  Theorem \ref{theorem6} $\ent(\overline{\beta}_K\restriction_{K^{(\Z)}}) = \ent(\sigma_\lambda)= \log |K|$.
\end{itemize}
\end{example}

\section{Applications of the main theorem}

We give now other consequences of Theorem \ref{theorem6}.
The first one is an application of Theorem \ref{theorem6} together with Remark \ref{add-rem}(b). It shows that, even if the restriction of $\sigma_\lambda$ to an invariant subgroup is not necessarily a generalized shift, its entropy obeys the same formula as the generalized shift does.

\begin{corollary}\label{lastcor}
Let $\Lambda\subseteq\Gamma$ be such that $\lambda^{-1}(\Lambda)\subseteq\Lambda$. Then $|s_\lambda|^*\log|K|=\ent(\sigma_\lambda\restriction_{G_\Lambda})+|s_{\lambda\restriction_{\Gamma\setminus\Lambda}}|^*\log|K|$. 
\end{corollary}
\begin{proof}
By \eqref{magic} in Remark \ref{add-rem}(b) and Theorem \ref{theorem6} $|s_\lambda|^*\log|K|=\ent(\sigma_\lambda)=\ent(\sigma_\lambda\restriction_{G_\Lambda})+\ent(\sigma_{\lambda\restriction_{\Gamma\setminus\Lambda}})=\ent(\sigma_\lambda\restriction_{G_\Lambda})+|s_{\lambda\restriction_{\Gamma\setminus\Lambda}}|^*\log|K|$.
\end{proof}

\begin{remark}\label{lastrem}
If $\Lambda\subseteq\Gamma$ is such that $\lambda^{-1}(\Lambda)\subseteq\Lambda$, and $S$ is a string of $\lambda$ with $S\cap\Lambda\neq\emptyset$, then $S\setminus\Lambda$ is finite and we can assume without loss of generality that $S\subseteq\Lambda$. Hence we can think that either $S\subseteq\Lambda$ or $S\subseteq\Gamma\setminus\Lambda$. This shows that:
\begin{itemize}
	\item[(a)] $s_{\lambda\restriction_{\Gamma\setminus\Lambda}}$ is the number of all strings of $\lambda$ which miss $\Lambda$, and
	\item[(b)] in case $s_\lambda$ is finite, $s_\lambda-s_{\lambda_{\Gamma\setminus\Lambda}}$ is the number of all strings of $\lambda$ contained in $\Lambda$.
\end{itemize}
\end{remark}

\begin{corollary}\label{theorem18}\label{cor7}\label{cor8}
\begin{itemize}
	\item[(a)] Let $\Lambda\subseteq\Gamma$ be such that $\lambda^{-1}(\Lambda)\subseteq\Lambda$.
If $\{\ent(\sigma_\lambda),\ent(\sigma_\lambda\restriction_{G_\Lambda})\}\cap\{0,+\infty\}=\emptyset$,
then $\ent(\sigma_\lambda\restriction_{G_\Lambda})/\ent(\sigma_\lambda)$ is rational.
Moreover, if $0<r\leq 1$ is a rational number such that $r(\ent(\sigma_\lambda)/\log |K|)\in \Z$, then there exists
$\Lambda\subseteq\Gamma$, such that $\ent(\sigma_\lambda\restriction_{G_\Lambda})/\ent(\sigma_\lambda)=r$.
	\item[(b)] Let $L$ be another finite abelian group and assume that both $K$ and $L$ have at least two elements.
Then $\log|L|\ent(\sigma_{\lambda,K})=\log|K|\ent(\sigma_{\lambda,L})$.
\end{itemize}
\end{corollary}
\begin{proof}
(a) By Theorem \ref{theorem6}, Corollary \ref{lastcor} and our assumption $\{\ent(\sigma_\lambda),\ent(\sigma_\lambda\restriction_{G_\Lambda})\}\cap\{0,+\infty\}=\emptyset$, 
$s_\lambda$ is finite and 
$\ent(\sigma_\lambda\restriction_{G_\Lambda})/\ent(\sigma_\lambda)=(s_\lambda-s_{\lambda_{\Gamma\setminus\Lambda}})/s_\lambda$ is a rational number. 

\smallskip
By hypothesis there exists $m\in\N$ such that $r s_\lambda=m$.
Let $S_1,\ldots,S_{s_\lambda}$ be the strongly disjoint strings in $\Gamma$ that realize $s_\lambda$ (this is possible by Claim \ref{strongly-disjoint} since $s_\lambda$ is finite by hypothesis).
Since $r\leq1$, it follows that $m\leq s_\lambda$. So it is possible to consider $T:=S_1\cup\ldots\cup S_m$ and define $\Lambda:= \bigcup_{n\in\N_0}\lambda^{-n}(T)$. Then $\lambda^{-1}(\Lambda)\subseteq\Lambda$. By Remark \ref{lastrem} $s_\lambda-s_{\lambda\restriction_{\Gamma\setminus\Lambda}}=m$, since $\{S_1,\ldots, S_m\}$ is a family of pairwise disjoint strings in $\Lambda$ of the maximal possible cardinality. By Corollary \ref{lastcor} $\ent(\sigma_\lambda\restriction_{G_\Lambda})=m\log|K|$. Then $\ent(\sigma_\lambda\restriction_{G_\Lambda})/\ent(\sigma_\lambda)=m/s_\lambda=r$.

\smallskip
(b) Is a simple application of Theorem \ref{theorem6}.
\end{proof}

\begin{corollary}\label{cor9}
For $M$ a torsion infinite abelian group and $\sigma_\lambda:M^{(\Gamma)}\to M^{(\Gamma)}$, $$\ent(\sigma_\lambda)=\begin{cases}
0 & \text{if}\ s_\lambda=0,\\
+\infty & \text{if}\ s_\lambda>0. 
\end{cases}$$
\end{corollary}
\begin{proof}
In view of Remark \ref{subgroup} and Theorem \ref{theorem6}, it is easy to see that 
\begin{align*}
\ent(\sigma_\lambda)&=\sup\{\ent(\sigma_\lambda\restriction_{H^{(\Gamma)}}):H\ \text{is a finite subgroup of}\ M\}\\
&=\sup\{\ent(\sigma_{\lambda,H}):H\ \text{is a finite subgroup of}\ M\}\\
&=\sup\{|s_\lambda|^*\log|H|:H\ \text{is a finite subgroup of}\ M\}.
\end{align*}
If $s_\lambda=0$, $\ent(\sigma_\lambda)=0$.
If $s_\lambda>0$, then $\ent(\sigma_\lambda)$ converges to $+\infty$ with $\log|H|$.
\end{proof}

\begin{example}\label{example10}
\begin{itemize}
\item[(a)]Let $\Gamma$ be a monoid.
\begin{itemize}
  \item[(a$_1$)] For each $s\in \Gamma$ consider $\lambda_s,\rho_s:\Gamma\rightarrow\Gamma$ defined by $\lambda_s(t)=st$ and $\rho_s(t)=ts$ for every $t\in\Gamma$. The element $s$ is invertible if and only if $\lambda_s$ and $\rho_s$ are bijective. By Example \ref{theorem3}, the endomorphisms $\sigma_{\lambda_s}$ and $\sigma_{\rho_s}$ of $G_\Gamma$ have the following properties.
    \begin{itemize}
        \item[(i)]If $s$ is of finite order $n\in\N$, then $(\lambda_s)^n=\lambda_{s^n}=id_\Gamma$ (i.e., $\lambda_s$ is periodic) and similarly $(\rho_s)^n=\rho_{s^n}=id_\Gamma$ (i.e., $\rho_s$ is periodic). By Proposition \ref{composition}(a) $\sigma_{\lambda_s}^n=\sigma_{\lambda_s^n}=\sigma_{\lambda_{s^n}}=id_{G_\Gamma}$ and $\sigma_{\rho_s}^n=\sigma_{\rho_s^n}=\sigma_{\rho_{s^n}}=id_{G_\Gamma}$. Hence $\ent(\sigma_{\lambda_s})=\ent(\sigma_{\rho_s})=0$ by Lemma \ref{locnilp}(b).
        \item[(ii)]Suppose that $s$ is invertible. If $s$ has infinite order, then $\lambda_s$ and $\rho_s$ are injective and by Example \ref{theorem4} their entropy is positive. So, $\ent(\sigma_{\lambda_s})=\ent(\sigma_{\rho_s})=0$ if and only if $s$ is of finite order.
    \end{itemize}
\item[(a$_2$)] For each invertible $s\in \Gamma$ consider $\mu_s:\Gamma\to\Gamma$ defined by $\mu_s(t)=s t s^{-1}$ for every $t\in\Gamma$. By Example \ref{theorem3}, the endomorphism $\sigma_{\mu_s}$ of $G_\Gamma$ has $\ent(\sigma_{\mu_s})>0$ if and only if there exists $t\in\Gamma$ such that $\{s^n:n\in\N\}\cap \{v\in\Gamma:vt=tv\}=\emptyset$.
\end{itemize}
	\item[(b)] Suppose now that $\Gamma$ is an abelian group and $\lambda:\Gamma\rightarrow\Gamma$ a group homomorphism such that $\ker\lambda$ is finite (i.e., $\lambda$ has finite fibers).
	\begin{itemize}
		\item[(b$_1$)]If $\Gamma=\Z$, there exists $n\in\Z$ such that $\lambda(x)=n x$ for every $x\in\Z$. If $n\neq\pm 1$, then there exists no string of $\lambda$ and so $\ent(\sigma_\lambda)=0$ by Theorem \ref{theorem6}. If $m=\pm 1$, then $\lambda^2=id_\Gamma$ and so by Proposition \ref{composition}(a) $\sigma_\lambda^2=\sigma_{\lambda^2}=\sigma_{id_\Gamma}=id_{G_\Gamma}$ and by Lemma \ref{locnilp}(b) $\ent(\sigma_\lambda)=0$.
		\item[(b$_2$)]Suppose that $\lambda\in\Aut(\Gamma)$. Then the orbits of $\lambda$ are exactly the equivalence classes of the relation $\Re_\lambda$ (see Example \ref{theorem3-first_part}). Therefore, if $\lambda$ has infinitely many infinite orbits, $\alpha_\lambda$ is infinite and by Example \ref{theorem3-first_part} and Theorem \ref{theorem6} $\ent(\sigma_\lambda)=+\infty$.
		\item[(b$_3$)]Consider $\Gamma=\Z\times\Z$ and $\lambda\in\Aut(\Gamma)$ defined by $\lambda(x,y)=(x+y,y)$ for every $(x,y)\in\Gamma$. For every $n\in\N$ the orbits of $(0,n)$, that is, $$(0,n)/\Re_\lambda=\{\ldots,(-2n,n),(-n,n),(0,n),(n,n),(2n,n),\ldots\},$$ are infinitely many and pairwise disjoint. Then $\ent(\sigma_\lambda)=+\infty$ by (b$_2$).
	\end{itemize}
\end{itemize}
\end{example}

The next example is dedicated to the composition of generalized shifts. Let us mention that from the formulas $\ent(\sigma_\lambda^k)=k \, \ent(\sigma_\lambda)$ (see Lemma \ref{log_law}) and $\sigma_\lambda^k= \sigma_{\lambda^k}$ (see Proposition \ref{composition}), and Theorem  \ref{theorem6} we obtain the useful non-obvious formula $s_{\lambda^k}=ks_{\lambda}$. 

\begin{example}
Let $\Gamma=\N_0$ and let $\mu_1:\Gamma\to\Gamma$ be defined by
\begin{equation*}
\mu_1(m)=
\begin{cases}
p^{2k} & \text{if}\ m=p^{2k+1}\ \text{with}\ p\in\mathbb P\ \text{and}\ k\in\N, \\
p^{2k+1} & \text{if}\ m=p^{2k}\ \text{with}\ p\in\mathbb P\ \text{and}\ k\in\N, \\
m & \text{otherwise}.
\end{cases}
\end{equation*}
Hence $\mu_1^2=id_\Gamma$ and by Proposition \ref{composition}(a) $\sigma_{\mu_1}^2=\sigma_{\mu_1^2}=\sigma_{id_\Gamma}=id_{G_\Gamma}$. Then $\ent(\sigma_{\mu_1\circ\mu_1})=0$ and since $\sigma_{\mu_1}$ is periodic, $\ent(\sigma_{\mu_1})=0$ by Lemma \ref{locnilp}(b).
The diagram for $\mu_1$ is the following:
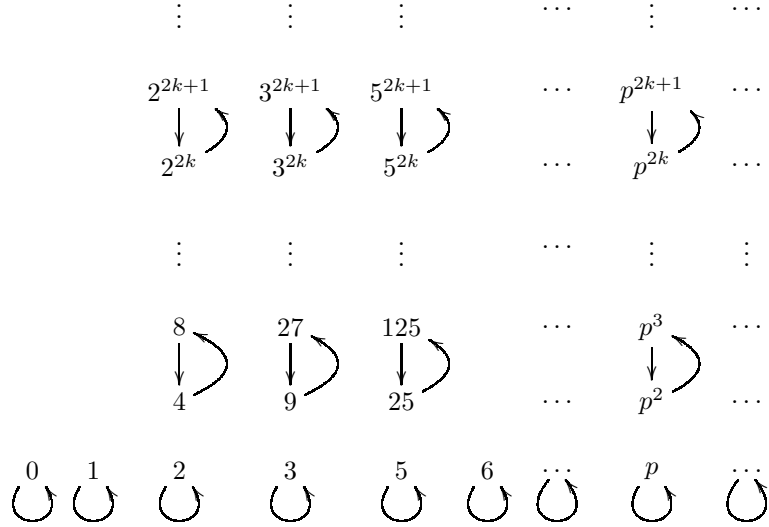
\begin{equation*}
\xymatrix@-1pc{
 & & \vdots & \vdots & \vdots & &\ldots & \vdots & \ldots\\
 &   &  2^{2k+1} \ar@{->}[d] & 3^{2k+1} \ar@{->}[d] & 5^{2k+1} \ar@{->}[d] & & \ldots &p^{2k+1} \ar@{->}[d]& \ldots\\
 &   & 2^{2k} \ar@{->}@/_1.5pc/[u] & 3^{2k} \ar@{->}@/_1.5pc/[u] & 5^{2k} \ar@{->}@/_1.5pc/[u] & & \ldots & p^{2k} \ar@{->}@/_1.5pc/[u]&\ldots \\
 &   & \vdots              & \vdots & \vdots & & \ldots & \vdots & \vdots \\
 &   & 8 \ar@{->}[d] & 27 \ar@{->}[d]& 125 \ar@{->}[d] & & \ldots & p^3 \ar@{->}[d] & \ldots\\
 &   & 4 \ar@{->}@/_1.5pc/[u] & 9 \ar@{->}@/_1.5pc/[u] & 25 \ar@{->}@/_1.5pc/[u] & & \ldots & p^2 \ar@{->}@/_1.5pc/[u]&\ldots\\
0 \ar@(dl,dr)[]& 1\ar@(dl,dr)[] & 2\ar@(dl,dr)[] & 3\ar@(dl,dr)[] & 5\ar@(dl,dr)[] & 6\ar@(dl,dr)[] &\ldots \ar@(dl,dr)[] & p\ar@(dl,dr)[] &\ldots\ar@(dl,dr)[]\\
}
\end{equation*}

\bigskip

\bigskip
Let $\mu_2:\Gamma\to\Gamma$ be defined by
\begin{equation*}
\mu_2(m)=
\begin{cases}
p^{2k-1} & \text{if}\ m=p^{2k}\ \text{with}\ p\in\mathbb P\ \text{and}\ k\in\N, \\
p^{2k} & \text{if}\ m=p^{2k-1}\ \text{with}\ p\in\mathbb P\ \text{and}\ k\in\N, \\
m & \text{otherwise (i.e., $m$ is not a prime power)}.
\end{cases}
\end{equation*}
Analogously, $\mu_2^2=id_\Gamma$ and by Proposition \ref{composition}(a) $\sigma_{\mu_2}^2=\sigma_{\mu_2^2}=\sigma_{id_\Gamma}=id_{G_\Gamma}$. Then $\ent(\sigma_{\mu_2\circ\mu_2})=0$ and sice $\sigma_{\mu_2}$ is periodic, so $\ent(\sigma_{\mu_2})=0$ by Lemma \ref{locnilp}(b).
The diagram for $\mu_2$ is the following:
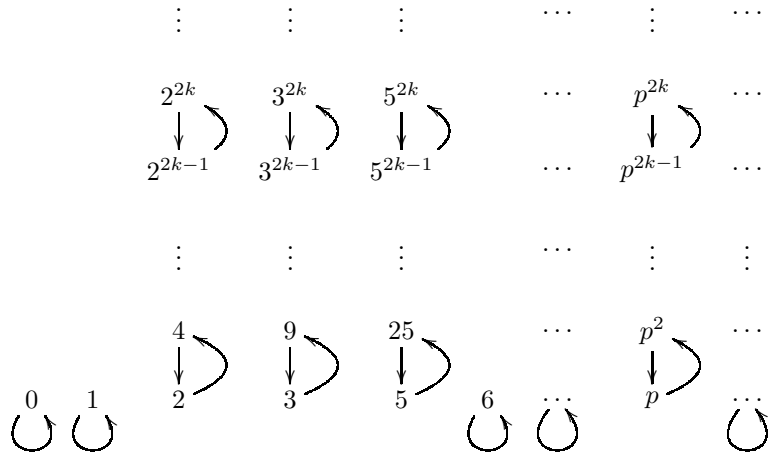
\begin{equation*}
\xymatrix@-1pc{
 & & \vdots & \vdots & \vdots & &\ldots & \vdots & \ldots\\
 &   &  2^{2k} \ar@{->}[d] & 3^{2k} \ar@{->}[d] & 5^{2k} \ar@{->}[d] & & \ldots &p^{2k} \ar@{->}[d]& \ldots\\
 &   & 2^{2k-1} \ar@{->}@/_1.5pc/[u] & 3^{2k-1} \ar@{->}@/_1.5pc/[u] & 5^{2k-1} \ar@{->}@/_1.5pc/[u] & & \ldots & p^{2k-1} \ar@{->}@/_1.5pc/[u]&\ldots \\
 &   & \vdots              & \vdots & \vdots & & \ldots & \vdots & \vdots \\
 &   & 4 \ar@{->}[d]  & 9 \ar@{->}[d] & 25 \ar@{->}[d] & & \ldots & p^2 \ar@{->}[d] &\ldots\\
0 \ar@(dl,dr)[]& 1\ar@(dl,dr)[] & 2\ar@{->}@/_1.5pc/[u] & 3\ar@{->}@/_1.5pc/[u] & 5\ar@{->}@/_1.5pc/[u] & 6\ar@(dl,dr)[] &\ldots \ar@(dl,dr)[] & p\ar@{->}@/_1.5pc/[u] &\ldots\ar@(dl,dr)[]\\
}
\end{equation*}

\bigskip

\bigskip
The diagram for $\mu_1\circ\mu_2:\Gamma\to\Gamma$ is the following:
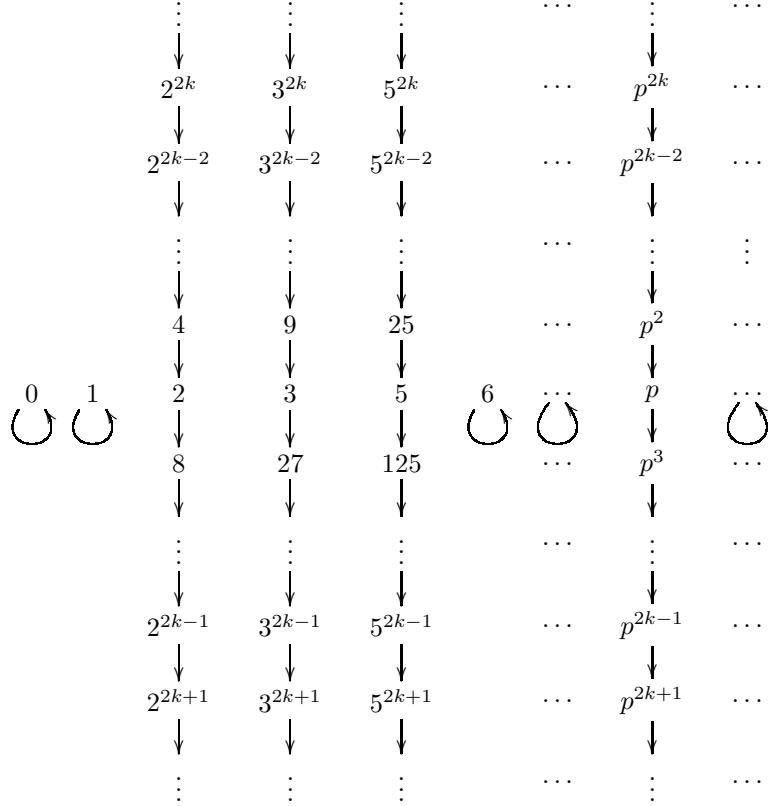
\begin{equation*}
\xymatrix@-1pc{
 & & \vdots\ar@{->}[d] & \vdots\ar@{->}[d] & \vdots\ar@{->}[d] & &\ldots & \vdots\ar@{->}[d] & \ldots\\
 &   &  2^{2k} \ar@{->}[d] & 3^{2k} \ar@{->}[d] & 5^{2k} \ar@{->}[d] & & \ldots & p^{2k} \ar@{->}[d]& \ldots\\
 &   & 2^{2k-2} \ar@{->}[d] & 3^{2k-2} \ar@{->}[d] & 5^{2k-2} \ar@{->}[d] & & \ldots & p^{2k-2} \ar@{->}[d]&\ldots \\
 &   & \vdots \ar@{->}[d] & \vdots\ar@{->}[d] & \vdots\ar@{->}[d] & & \ldots & \vdots\ar@{->}[d] & \vdots \\
 &   & 4 \ar@{->}[d]  & 9 \ar@{->}[d] & 25 \ar@{->}[d] & & \ldots & p^2 \ar@{->}[d] &\ldots\\
0 \ar@(dl,dr)[]& 1\ar@(dl,dr)[] & 2\ar@{->}[d] & 3\ar@{->}[d] & 5\ar@{->}[d] & 6\ar@(dl,dr)[] &\ldots \ar@(dl,dr)[] & p\ar@{->}[d] &\ldots\ar@(dl,dr)[]\\
 &   & 8 \ar@{->}[d] & 27 \ar@{->}[d] & 125 \ar@{->}[d] & & \ldots & p^{3} \ar@{->}[d]&\ldots \\
 &   & \vdots\ar@{->}[d] & \vdots\ar@{->}[d] & \vdots \ar@{->}[d] & &\ldots & \vdots\ar@{->}[d] & \ldots\\
 &   &  2^{2k-1} \ar@{->}[d] & 3^{2k-1} \ar@{->}[d] & 5^{2k-1} \ar@{->}[d] & & \ldots & p^{2k-1} \ar@{->}[d]& \ldots\\
 &   &  2^{2k+1} \ar@{->}[d] & 3^{2k+1} \ar@{->}[d] & 5^{2k+1} \ar@{->}[d] & & \ldots & p^{2k+1} \ar@{->}[d]& \ldots\\
 & & \vdots & \vdots & \vdots & &\ldots & \vdots & \ldots\\
}
\end{equation*}
In this case $s_{\mu_1\circ\mu_2}=\alpha_{\mu_1\circ\mu_2}=\omega$ (so $|s_{\mu_1\circ\mu_2}|^*=+\infty$) and by Theorem \ref{theorem6} $\ent(\sigma_{\mu_1\circ\mu_2})=+\infty$.
Similarly one can see that also s$_{\mu_2 \circ \mu_1}=\alpha_{\mu_1\circ\mu_2}=\omega$ and so that $\ent(\sigma_{\mu_2\circ\mu_1})=+\infty$.

By Proposition \ref{composition}(a) $\sigma_{\mu_1\circ\mu_2}=\sigma_{\mu_1}\circ\sigma_{\mu_2}$; hence $\sigma_{\mu_1}\circ\sigma_{\mu_2}$ is an example of an endomorphism of infinite entropy with both $\sigma_{\mu_1}$ and $\sigma_{\mu_2}$ of entropy $0$. The same is $\sigma_{\mu_2}\circ\sigma_{\mu_1}$.

\bigskip

Let $\varrho_1:\Gamma\to\Gamma$ be defined by $\varrho_1(0)=1$ and $\varrho_1(m)=m$ for every $m\in\N$. Then $s_{\varrho_1}=0$ and so $\ent(\sigma_{\varrho_1})=0$ by Theorem \ref{theorem6}.

\medskip
The diagram for $\varrho_1\circ\varphi_1:\Gamma\to\Gamma$ is the following:
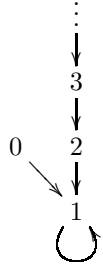
\begin{equation*}
\xymatrix@-1pc{
  & \vdots\ar@{->}[d]\\  
  & 3\ar@{->}[d]  \\
0 \ar@{->}[dr] & 2 \ar@{->}[d]\\
 & 1\ar@(dl,dr)[]
}
\bigskip
\end{equation*}

\bigskip
For this function $s_{\varrho_1\circ\varphi_1}=1$, and so by Theorem \ref{theorem6} $\ent(\sigma_{\varrho_1\circ\varphi_1})=\log|K|$.

By Proposition \ref{composition}(a) $\sigma_{\varrho_1\circ\varphi_1}=\sigma_{\varrho_1}\circ\sigma_{\varphi_1}$; consequently $\ent(\sigma_{\varrho_1}\circ\sigma_{\varphi_1})=\ent(\sigma_{\varphi_1})=\log|K|$, while $\ent(\sigma_{\varrho_1})=0$.
\end{example}

\begin{theorem}\label{theorem14}
If $\lambda:\Gamma\to\Gamma$, $\mu:\Upsilon\to\Upsilon$ are such that for each
$(i,j)\in\Gamma\times\Upsilon$, $\lambda^{-1}(i)$ and
$\mu^{-1}(j)$ are finite, then for the endomorphism $\sigma_{\lambda\times\mu}:G_{\Gamma\times\Upsilon}\to G_{\Gamma\times\Upsilon}$ we have:
\begin{equation*}
\ent(\sigma_{\lambda\times\mu})=0\ \text{\emph{if}}\  \begin{cases}
\ent(\sigma_{\lambda})=\ent(\sigma_{\mu})=0, & \text{\emph{(a$_0$)}}\\
\ent(\sigma_\mu)=|\Per(\mu)|=0\ \text{and}\ \ent(\sigma_\lambda)>0, & \text{\emph{(a$_1$)}}\\
\ent(\sigma_\lambda)=|\Per(\lambda)|=0\ \text{and}\ \ent(\sigma_\mu)>0. &  \text{\emph{(a$_2$)}}
\end{cases}\eqno{(a)}
\end{equation*}
Moreover,
\begin{equation*}
\ent(\sigma_{\lambda\times\mu})=
\begin{cases}
+\infty & \text{\emph{if}}\ \ent(\sigma_{\lambda})>0\ \text{and} \ \ent(\sigma_{\mu})>0,\ \ \ \ \ \ \ \ \ \ \ \ \ \ \ \ \ \ \ \text{\emph{(b$_0$)}}\\
|\Per(\mu)|^*\ent(\sigma_\lambda) & \text{\emph{if}}\ \ent(\sigma_\mu)=0,\ |\Per(\mu)|>0\ \text{and}\ \ent(\sigma_\lambda)>0,\ \ \ \text{\emph{(b$_1$)}}\\
|\Per(\lambda)|^*\ent(\sigma_\mu) & \text{\emph{if}}\ \ent(\sigma_\lambda)=0,\ |\Per(\lambda)|>0\ \text{and}\ \ent(\sigma_\mu)>0.\ \ \ \text{\emph{(b$_2$)}}
\end{cases}\eqno{(b)}
\end{equation*}
\end{theorem}
\begin{proof}
If $\ent(\sigma_{\lambda\times\mu})>0$, then by Theorem \ref{theorem6}, there exists a string
$\{(m_t,n_t)\}_{-t\in\N_0}$ of $\lambda\times\mu$ (for each $-t\in\N_0$, $(m_{t+1},n_{t+1})=(\lambda(m_t),\mu(n_t))$),
therefore at least one of the sequences $\{m_t\}_{-t\in\N_0}$ or $\{n_t\}_{-t\in\N_0}$ is a string, which shows that either
$\ent(\sigma_\lambda)>0$ or $\ent(\sigma_\mu)>0$, in view of Theorem \ref{theorem6}. This proves (a$_0$).

\smallskip
Assume that $\ent(\sigma_\mu)=0$ and $\ent(\sigma_\lambda)>0$. We prove that $\ent(\sigma_{\lambda\times\mu})=0$ if $|\Per(\mu)|=0$. To this end, suppose that $\ent(\sigma_{\lambda\times\mu})>0$. By Theorem \ref{theorem6} $s_{\lambda\times\mu}>0$, so let $\{(m_t,n_t)\}_{-t\in\N_0}$ be a string of $\lambda\times\mu$. Since $\ent(\sigma_\mu)=0$, $s_\mu=0$ by Theorem \ref{theorem6}, and so $\{m_t\}_{-t\in\N_0}$ has to be a string  of $\lambda$ and $n_0$ is a periodic point of $\mu$. In particular $|\Per(\mu)|>0$. This proves (a$_1$). 

Reverting the roles of $\lambda$ and $\mu$ one can prove (a$_2$).

\smallskip
Now let $\ent(\sigma_\lambda)>0$ and $\ent(\sigma_\mu)>0$. By Theorem \ref{theorem6} there exist strings $\{m_t\}_{-t\in\N_0}$ and $\{n_t\}_{-t\in\N_0}$ respectively of $\lambda$ and $\mu$. For each $-l\in\N_0$, let $z_{l}:=(m_0,n_l)$. Then $\{(m_t,n_{l+t})\}_{-t\in\N_0}$ is a string of $\lambda\times\mu$ for every $l\in\N$, and these strings are pairwise disjoint. This means that $|s_{\lambda\times\mu}|^*=+\infty$ and by Theorem \ref{theorem6} $\ent(\sigma_{\lambda\times\mu})=+\infty$. This proves (b$_0$).

\smallskip
Assume that $\ent(\sigma_\mu)=0$, $|\Per(\mu)|>0$ and $\ent(\sigma_\lambda)>0$. We prove that
\begin{equation}\label{sigmaXmu}
\ent(\sigma_{\lambda\times\mu})\leq |\Per(\mu)|^* \ent(\sigma_\lambda). 
\end{equation}
If $\ent(\sigma_{\lambda\times\mu})=0$ the inequality in \eqref{sigmaXmu} is trivially satisfied. So we can assume that $\ent(\sigma_{\lambda\times\mu})>0$. By Theorem \ref{theorem6} $s_{\lambda\times\mu}>0$. Let $\{(m_t,n_t)\}_{-t\in\N_0}$ be a string  of $\lambda\times\mu$. Since $s_\mu=0$ by Theorem \ref{theorem6}, $\{m_t\}_{-t\in\N_0}$ has to be a string  of $\lambda$ and each $n_t$ is a periodic point of $\mu$. If $\{(m_t,n_t)\}_{-t\in\N_0}$ and $\{(m'_t,n'_t)\}_{-t\in\N_0}$ are disjoint strings of $\lambda\times\mu$, then either $\{m_t\}_{-t\in\N_0}$ and $\{m'_t\}_{-t\in\N_0}$ are disjoint strings of $\lambda$ or $n_0$ and $n'_0$ are distinct periodic points of $\mu$. This proves that $s_{\lambda\times\mu}\leq |\Per(\mu)| s_\lambda$. In particular \eqref{sigmaXmu} holds by Theorem \ref{theorem6}.

We show now that under the same hypotheses, also the converse implication holds true, that is, we prove that
\begin{equation}\label{sigmaXmu2}
\ent(\sigma_{\lambda\times\mu})\geq|\Per(\mu)|^*\ent(\sigma_\lambda).
\end{equation}
If $\{m_t\}_{-t\in\N_0}$ is a string of $\lambda$, and $j\in\Per(\mu)$ with $\{j=j_0,j_1,j_2,\ldots, j_s\}$ the finite orbit of $j$ (i.e., $\mu(j_k)=j_{k+1}$ for every $k\in\{0,\ldots,s-1\}$ and $\mu(j_s)=j$), then $\{(m_t,j_{[t]_{s+1}})\}_{-t\in\N_0}$ is a string of $\lambda\times\mu$ (where, for $a\in\Z,\ b\in\N$, $[a]_{b}$ denotes the remainder class of $a$ modulo $b$). In case $i,j$ are distinct elements of $\Per(\mu)$, the strings $\{(m_t,j_{[t]_{s+1}})\}_{-t\in\N_0}$ and $\{(m_t,i_{[t]_{r+1}})\}_{-t\in\N_0}$ (where $\{i=i_0,i_1,i_2,\ldots, i_r\}$ is the finite orbit of $i$, i.e., $\mu(i_k)=i_{k+1}$ for every $k\in\{0,\ldots,r-1\}$ and $\mu(i_r)=i$) are pairwise disjoint. This proves that $s_{\lambda\times\mu}\geq|\Per(\mu)|s_\lambda$, and by Theorem \ref{theorem6} \eqref{sigmaXmu2} holds.

By \eqref{sigmaXmu} and \eqref{sigmaXmu2} $\ent(\sigma_{\lambda\times\mu})=|\Per(\mu)|^*\ent(\sigma_\lambda)$. This concludes the proof of (b$_1$).

Reverting the roles of $\lambda$ and $\mu$ one can prove (b$_2$).
\end{proof}

In the notations of this theorem, the following example shows that in the case where $\ent(\sigma_\mu)=0$ and $\ent(\sigma_\lambda)=\log|K|$, it is possible that $\ent(\sigma_{\lambda\times \mu})$ is positive and also infinite, depending on the cardinality of $\Per(\mu)$.

\begin{example}\label{example15}
Let $t\in\N$, $\Gamma_t=\{1,\ldots,t\}$ and $\theta_t=(123\dots t)\in S_{\Gamma_t}$. Let $\Lambda=\N_0\times\Gamma_t$. Then $\varphi_1\times\theta_t:\Lambda\to\Lambda$ and its diagram is the following:
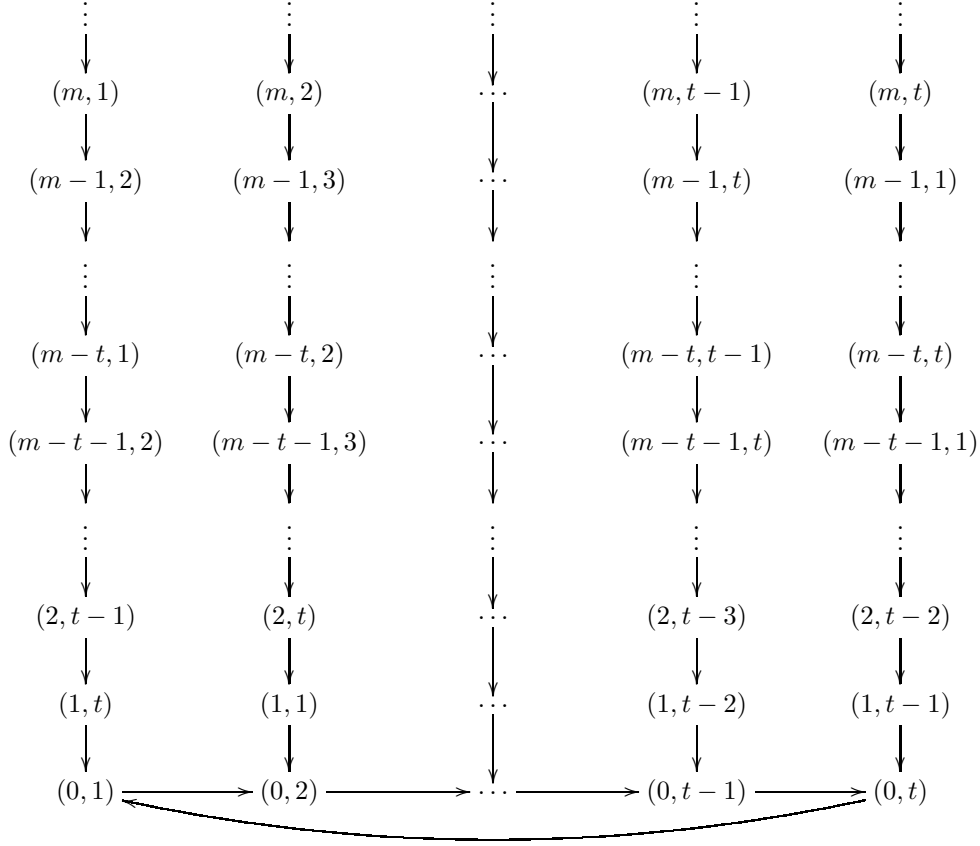
\begin{equation*}
\xymatrix@!R@!C@-1pc{
\vdots \ar@{->}[d] & \vdots \ar@{->}[d] & \vdots \ar@{->}[d] & \vdots \ar@{->}[d] & \vdots \ar@{->}[d] \\
(m,1) \ar@{->}[d] & (m,2) \ar@{->}[d] & \ldots \ar@{->}[d] & (m,t-1) \ar@{->}[d] & (m,t) \ar@{->}[d] \\
(m-1,2) \ar@{->}[d] & (m-1,3) \ar@{->}[d] & \ldots \ar@{->}[d] & (m-1,t) \ar@{->}[d] & (m-1,1) \ar@{->}[d] \\
\vdots \ar@{->}[d] & \vdots \ar@{->}[d] & \vdots \ar@{->}[d] & \vdots \ar@{->}[d] & \vdots \ar@{->}[d] \\
(m-t,1) \ar@{->}[d] & (m-t,2) \ar@{->}[d] & \ldots \ar@{->}[d] & (m-t,t-1) \ar@{->}[d] & (m-t,t) \ar@{->}[d] \\
(m-t-1,2)\ar@{->}[d]&(m-t-1,3)\ar@{->}[d]&\ldots\ar@{->}[d]&(m-t-1,t)\ar@{->}[d]&(m-t-1,1)\ar@{->}[d] \\
\vdots\ar@{->}[d] & \vdots\ar@{->}[d] & \vdots\ar@{->}[d] & \vdots\ar@{->}[d] & \vdots\ar@{->}[d] \\
(2,t-1) \ar@{->}[d] & (2,t) \ar@{->}[d] & \ldots \ar@{->}[d] & (2,t-3)\ar@{->}[d] & (2,t-2) \ar@{->}[d] \\
(1,t) \ar@{->}[d] & (1,1) \ar@{->}[d] & \ldots \ar@{->}[d] & (1,t-2) \ar@{->}[d] & (1,t-1) \ar@{->}[d] \\
(0,1) \ar@{->}[r] & (0,2) \ar@{->}[r] & \ldots \ar@{->}[r] & (0, t-1)\ar@{->}[r] & (0,t) \ar@/^1.5pc/@{->}[llll] \\
}
\end{equation*}

\bigskip
Let $\theta=(12)\in S_\N$ and let $\Lambda=\N\times \N_0$. The diagram for $\theta\times \varphi_1:\Lambda\to\Lambda$ is the following:
\begin{equation*}
\xymatrix@!R@!C@-1pc{
\vdots \ar@{->}[d] & \vdots \ar@{->}[d] & \vdots \ar@{->}[d] & \ldots \ar@{->}[d] & \vdots \ar@{->}[d] & \vdots \ar@{->}[d] \\
(1,2) \ar@{->}[d] & (2,2) \ar@{->}[d] & (3,2) \ar@{->}[d] & \ldots \ar@{->}[d] & (n,2) \ar@{->}[d] & \ldots \ar@{->}[d] \\
(2,1) \ar@{->}[d] & (1,1) \ar@{->}[d] & (3,1) \ar@{->}[d] & \ldots \ar@{->}[d] & (n,1) \ar@{->}[d] & \ldots\ar@{->}[d] \\
(1,0) \ar@{->}[r] & (2,0) \ar@{->}@/^1.5pc/[l] & (3,0)\ar@(dl,dr)[] & \ldots\ar@(dl,dr)[] & (n,0)\ar@(dl,dr)[] & \ldots \ar@(dl,dr)[]\\
}
\bigskip
\end{equation*}

Then by Theorem \ref{theorem6} $\ent(\sigma_{\varphi_1\times\theta_t})=t\log|K|$ and $\ent(\sigma_{\varphi_1\times\theta})=+\infty$, while $\ent(\sigma_{\varphi_1})=\log|K|$ and $\ent(\sigma_{\theta})=\ent(\sigma_{\theta_t})=0$, since $|\Per(\theta_t)|=t$, $|\Per(\theta)|^*=+\infty$. 
\end{example}

The next is an application of Theorem \ref{theorem14} and Corollary \ref{cor7}. Indeed considering the product of two finite abelian groups $K\times L$ instead of only one finite abelian group $K$ (as in Theorem \ref{theorem14}) is not a more general situation, since Theorem \ref{theorem6}, but also Corollary \ref{cor7}, shows that the entropy of a generalized shift depends mainly on its string number $s_\lambda$.

\begin{corollary}\label{cor16}
Let $K$ and $L$ be finite non-trivial abelian groups and let $\lambda:\Gamma\rightarrow\Gamma$ and
$\mu:\Upsilon\rightarrow\Upsilon$ be such that for each $(i,j)\in\Gamma\times\Upsilon$, $\lambda^{-1}(i)$ and
$\mu^{-1}(j)$ are finite. Then for $\sigma_{\lambda\times\mu}:(K\times L)^{(\Gamma\times\Upsilon)}\to(K\times L)^{(\Gamma\times\Upsilon)}$ we have:
\begin{equation*}
\ent(\sigma_{\lambda\times\mu})=0\ \text{\emph{if}}\ \begin{cases}
\ent(\sigma_{\lambda})=\ent(\sigma_{\mu})=0, & \text{\emph{(a$_0$)}}\\
\ent(\sigma_\mu)=|\Per(\mu)|=0\ \text{and}\ \ent(\sigma_\lambda)>0, & \text{\emph{(a$_1$)}}\\
\ent(\sigma_\lambda)=|\Per(\lambda)|=0\ \text{and}\ \ent(\sigma_\mu)>0. & \text{\emph{(a$_2$)}}
\end{cases}
\eqno{(a)}\end{equation*}
Moreover,
\begin{equation*}
\ent(\sigma_{\lambda\times\mu})=
\begin{cases}
+\infty & \text{\emph{if}}\ \ent(\sigma_{\lambda})>0\ \text{and} \ \ent(\sigma_{\mu})>0,\ \ \ \ \ \ \ \ \ \ \ \ \ \ \ \ \ \ \ \text{\emph{(b$_0$)}}\\
\frac{\log|K\times L|}{\log|K|}|\Per(\mu)|^*\ent(\sigma_\lambda) & \text{\emph{if}}\ \ent(\sigma_\mu)=0,\ |\Per(\mu)|>0\ \text{and}\ \ent(\sigma_\lambda)>0,\ \ \ \text{\emph{(b$_1$)}}\\
\frac{\log|K\times L|}{\log|K|}|\Per(\lambda)|^*\ent(\sigma_\mu) & \text{\emph{if}}\ \ent(\sigma_\lambda)=0,\ |\Per(\lambda)|>0\ \text{and}\ \ent(\sigma_\mu)>0.\ \ \ \text{\emph{(b$_2$)}}
\end{cases}\eqno{(b)}
\end{equation*}
\end{corollary}
\begin{proof}
By Corollary \ref{cor7} $\ent(\sigma_{\lambda\times\mu,K\times L})=\frac{\log|K\times L|}{\log|K|}\ent(\sigma_{\lambda\times\mu,K})$. Now apply Theorem \ref{theorem14}.
\end{proof}

In the next example we associate to a given map $\lambda$ a natural extension map $\Lambda$ such that $\ent(\sigma_\Lambda)$ is either infinite or $0$, depending on whether $\ent(\sigma_\lambda)$ is positive or $0$.

\begin{example}\label{example17}
Define $\Lambda:\mathcal P_{fin}(\Gamma)\rightarrow{\mathcal P}_{fin}(\Gamma)$ by $\Lambda(A)=\{\lambda(i):i\in A\}=\lambda(A)$ for every $A\in {\mathcal P}_{fin}(\Gamma)$.
(Since $\Gamma$ embeds into ${\mathcal P}_{fin}(\Gamma)$ in a natural way via the singletons, $\Lambda$ can be considered as an extension of $\lambda$.) 
For each $A\in\mathcal P_{fin}(\Gamma)$, $\Lambda^{-1}(A)$ is a finite subset of ${\mathcal P}_{fin}(\Gamma)$. Consider $\sigma_\Lambda:G_{\mathcal P_{fin}(\Gamma)}\to G_{\mathcal P_{fin}(\Gamma)}$. Then:
\begin{equation*}
\ent(\sigma_{\Lambda})=
\begin{cases}
    0 & \text{if}\ \ent(\sigma_\lambda)=0, \\
    +\infty & \text{if}\ \ent(\sigma_\lambda)>0.
\end{cases}
\end{equation*}
If $\ent(\sigma_\Lambda)>0$, by Theorem \ref{theorem6} there exists at least one string $\{A_t\}_{-t\in\N_0}$ of $\Lambda$ in $\mathcal P_{fin}(\Gamma)$. In particular there exists a string $\{m_t\}_{-t\in\N_0}$ of $\lambda$ in $\Gamma$: indeed, there exists $-t\in\N_0$ such that not all elements of $A_t$ are periodic for $\lambda$. Suppose without loss of generality that $t=0$ and let $m_{0}\in A_{0}\setminus \Per(\lambda)$. Then there exists an infinite sequence $\{m_t\}_{-t\in\N_0}$ of elements of $\Gamma$ such that $m_t\in A_t$ and $\lambda(m_t)=m_{t+1}$ for every $-t\in\N$. The elements $m_t$ have to be distinct because $m_{0}$ is not periodic. So $\{m_t\}_{-t\in\N_0}$ is a string of $\lambda$ and by Theorem \ref{theorem6} $\ent(\sigma_\lambda)>0$. This proves that if $\ent(\sigma_\lambda)=0$ then $\ent(\sigma_\Lambda)=0$.

If $\ent(\sigma_\lambda)>0$, by Theorem \ref{theorem6} there exists a string $\{m_t\}_{-t\in\N_0}$ of $\lambda$ in $\Gamma$. For each $l\in\N_0$,
$S_l:=\{\{m_t,m_{t-1},\ldots,m_{t-l}\}\}_{-t\in\N_0}$ is a string of $\Lambda$ in $\mathcal P_{fin}(\Gamma)$ and obviously $S_1,\ldots,S_l,\ldots$ are pairwise disjoint strings of $\Lambda$ in $\mathcal P_{fin}(\Gamma)$. So Theorem \ref{theorem6} leads us to the desired result, that is $\ent(\sigma_\Lambda)=+\infty$.
\end{example}

\section{Final remarks and open problems}

We consider here the generalized shift $\sigma_\lambda$  on $K^\Gamma$ and, in case $\lambda$ has finite fibers,  its restriction $\sigma_{\lambda}^\oplus$ on $K^{(\Gamma)}$.
Theorem \ref{theorem6} calculates precisely the value of the entropy of $\sigma_{\lambda}^\oplus:K^{(\Gamma)}\to K^{(\Gamma)}$. The necessary property on $\lambda$ to have finite fibers helps us in finding the explicit formula for the entropy of $\sigma_{\lambda}^\oplus$. In the general case of $\sigma_\lambda:K^\Gamma\to K^\Gamma$ we 
leave open the following problem.

\begin{problem}
Calculate the entropy of $\sigma_\lambda:K^\Gamma\to K^\Gamma$. Is $\ent(\sigma_\lambda)=\ent(\sigma_{\lambda}^\oplus)$ in case $\lambda$ has finite fibers?
\end{problem}

Note that $\ent(\sigma_\lambda)\geq\ent(\sigma_{\lambda}^\oplus)$ in the latter case, since then $K^{(\Gamma)}$ is $\sigma_\lambda$-invariant in $K^\Gamma$.

\begin{problem}
\begin{itemize}
	\item[(a)] Is it possible to have $\ent(\sigma_{\lambda}^\oplus)=0$, but $\ent(\sigma_\lambda)>0$?
	\item[(b)] Is it possible to have $0<\ent(\sigma_\lambda)<+\infty$?
\end{itemize}
\end{problem}

Moreover, Theorem \ref{theorem6} concerns only a single abelian group $K$. If $\{K_i:i\in\Gamma\}$ is a family of abelian groups and $\rho_i:K_{\lambda(i)}\to K_i$ is a homomorphism for every $i\in I$, define $\widetilde\sigma_\lambda:\prod_{i\in\Gamma}K_i\to\prod_{i\in\Gamma}K_i$ by $\widetilde\sigma_\lambda(x)=(\rho_i(x_{\lambda(i)}))_{i\in\Gamma}$ for every $x=(x_i)_{i\in\Gamma}\in\prod_{i\in\Gamma}K_i$. It is possible to consider also $\widetilde\sigma_\lambda$ restricted to the direct sum, that is, $\widetilde\sigma_{\lambda}^\oplus:\bigoplus_{i\in\Gamma}K_i\to\bigoplus_{i\in\Gamma}K_i$.

\begin{problem}
Suppose that for every $i\in\Gamma$ $K_i$ is a finite abelian group.
\begin{itemize}
	\item[(a)]Calculate the entropy of $\widetilde\sigma_\lambda:\prod_{i\in\Gamma}K_i\to\prod_{i\in\Gamma}K_i$.
	\item[(b)]Calculate the entropy of $\widetilde\sigma_{\lambda}^\oplus:\bigoplus_{i\in\Gamma}K_i\to\bigoplus_{i\in\Gamma}K_i$.
\end{itemize}
\end{problem}

A particular case of this problem is when for each $i\in\Gamma$ $K_{\lambda(i)}\leq K_i$, that is, $\rho_i:K_{\lambda(i)}\to K_i$ is an injective homomorphism. So one can consider first the problem in this case.

\medskip
Let $K$ be a finite field and $R=K[x]$. For $r\in R$, let $m_r:R\to R$ be defined by $m_r(s)=r s$ for every $s\in R$. It is easy to see that for the 
natural isomorphism $j:K[x]\to K^{(\N_0)}$ the conjugated isomorphism $j \circ m_x \circ  j^{-1}$ coincides with the right Bernoulli shift $\beta_K$ of $K^{(\N_0)}$ (and consequently, 
$j \circ m_{x^n} \circ  j^{-1}=\beta^n_K$).
 Therefore, the endomorphism $m_r$ is (conjugated to) a linear combination of powers of the Bernoulli shift  $\beta_K$.

\begin{problem}
Calculate the entropy of $m_r:K[x]\to K[x]$. What about the ring $K[x_1,\ldots,x_n]$ of polynomials in more variables? 
\end{problem}

This problem can be generalized for graded rings. (Let us recall that a \emph{graded ring} is a ring $R$ with a family $\{R_i:i\in\N_0\}$ of subgroups of $(R,+)$ such that $R=\bigoplus_{i=0}^\infty R_i$ and $R_i R_j\subseteq R_{i+j}$ for all $i,j\in \N_0$ \cite[Chapter 10]{AM}.)
For $r\in R$, let $m_r:R\to R$ be defined by $m_r(s)=r s$ for every $s\in R$. 

\begin{problem}\label{Last_problem}
Compute the entropy $\ent(m_r)$ in case $R$ is a graded ring, $r\in R$ and $m_r:R\to R$.
\end{problem}

Problem \ref{Last_problem} can be extended also to graded $R$-modules $M$ and the endomorphism $m_r$ of $M$ defined by the multiplication, in $M$,  by  a fixed element $r\in R$ as above.   

\medskip
We conclude with a problem suggested by Example \ref{example10}(b).

\begin{problem}
Let $\Gamma$ be an abelian group and $\lambda:\Gamma\to\Gamma$ an endomorphism.
\begin{itemize}
\item[(a)] Is it true that $s_\lambda>0$ implies $s_\lambda$ infinite?
\item[(b)] Describe in which cases $s_\lambda=0$ and in which cases $s_\lambda$ is infinite.
\end{itemize}
\end{problem}


\end{document}